\documentclass{amsart}

\usepackage[all]{xy}
\usepackage{amsmath}
\usepackage{hyperref}
\hypersetup{
    colorlinks=true,    
    linkcolor=blue,          
    citecolor=blue,      
    filecolor=blue,      
    urlcolor=blue           
}
\usepackage{tikz}
\usetikzlibrary{graphs,positioning,arrows,shapes.misc,decorations.pathmorphing}

\tikzset{
    >=stealth,
    every picture/.style={thick},
    graphs/every graph/.style={empty nodes},
}

\tikzstyle{vertex}=[
    draw,
    circle,
    fill=black,
    inner sep=1pt,
    minimum width=5pt,
]
\usepackage[position=top]{subfig}
\usepackage[alphabetic,backrefs]{amsrefs}
\usepackage{amssymb}
\usepackage{color}

\calclayout

\usetikzlibrary{decorations.pathmorphing}
\tikzstyle{printersafe}=[decoration={snake,amplitude=0pt}]

\newcommand{\pp}{\mathbb{P}}

\renewcommand{\qq}{\mathbb{Q}}
\newcommand{\zz}{\mathbb{Z}}
\newcommand{\nn}{\mathbb{N}}

\newtheorem{introthm}{Theorem}

\newtheorem{introcor}{Corollary}

\newtheorem{theorem}{Theorem}[section]
\newtheorem{lemma}[theorem]{Lemma}
\newtheorem{proposition}[theorem]{Proposition}
\newtheorem{corollary}[theorem]{Corollary}

\newtheorem{notation}[theorem]{Notation}
\newtheorem{definition}[theorem]{Definition}
\newtheorem{example}[theorem]{Example}

\newtheorem{remark}[theorem]{Remark}

\theoremstyle{remark}

\numberwithin{equation}{section}

\begin{document}

\title[Boudnedness of cone singularities]{A boundedness theorem for cone singularities}

\author[J.~Moraga]{Joaqu\'in Moraga}
\address{
Department of Mathematics, University of Utah, 155 S 1400 E, JWB 321,
Salt Lake City, UT 84112, USA}
\email{moraga@math.utah.edu}

\subjclass[2010]{Primary 14E30, 
Secondary 14M25.}
\maketitle

\begin{abstract}
A cone singularity is a normal affine variety $X$ with an effective one-dimensional torus action with a unique fixed point $x\in X$ which lies in the closure of any orbit of the $k^*$-action. In this article, we prove a boundedness theorem for cone singularities in terms of their dimension, singularities, and isotropies.
Given $d$ and $N$ two positive integers and $\epsilon$ a positive real number, we prove that the class 
of $d$-dimensional $\epsilon$-log canonical cone singularities with isotropies bounded by $N$ forms a bounded family.
\end{abstract}

\tableofcontents

\section{Introduction}

In algebraic geometry we are mostly interested on smooth varieties.
However, singular varieties appear naturally when studying smooth objects.
For instance, in Mori's theory singularities appear when running a minimal model program (see, e.g.~\cite{KM98,HK10}),
and singularities also appear on the Gromov-Hausdorff limit of a sequence of K\"ahler-Einstein manifolds (see, e.g.~\cite{DS14,DS16}).
Since the introduction of the minimal model program its been clear that certain classes of singularities are of special importance for birational geometers~\cite{Kol13}.
The development of projective geometry has been tangled with the understanding of the theory of singularities~\cite{Xu17}.
Indeed, Kawamata log terminal singularities, which are the main class of singularities on the MMP, are a local analogue of Fano varieties.
Unfortunately, a complete characterization of klt singularities in dimensions greater or equal than four seems to be unfeasable~\cite{Kol11}.
However, it is expected that the boundedness of Fano varieties due to Birkar~\cite{Bir16a,Bir16b} 
will have a vast number of applications to the understanding of klt singularites.
In this article, we investigate an application of such boundedness result to the study of the so-called cone singularities (see, e.g.~\cite{LS13}).
Cone singularities appear naturally in many contexts of algebraic geometry:
Toric geometry~\cite{CLS11,Ful93}, $\mathbb{T}$-varieties~\cite{AH06,AHS08,AIPSV12}, terminal $3$-fold singularities~\cite{Hay05a,Hay05b,Tzi05}, stability theory of klt singularities~\cite{LX16,LX17,Li17}, the graded ring of a valuation~\cite{Tei99}, 
Gromov-Hausdorff limits~\cite{DS14,DS16}, and Cox rings of Fano varieties~\cite{Bro13,GOST15}, among many others.
A cone over a projective variety $Y$ is a local version of such variety, and the global geometry of $Y$ is reflected on the singularity at the vertex of the cone.
For instance, the cone singularity is Kawamata log terminal if and only if the corresponding normalized Chow quotient $Y$ is of Fano type.
Henceforth, it is expected that the boundedness of Fano varieties implies the boundedness of some klt cone singularities,
and more generally the boundedness of certain klt singularities.
In this article, we give a first step in this direction proving that some natural class of klt cone singularities forms a bounded family.

A cone singularity is a normal affine variety $X$ with an effective one-dimensional torus action with a unique fixed point $x\in X$ which lies in the closure of any orbit of the $k^*$-action.
The fixed point for the torus action is often called the vertex of the cone singularity.
We say that the cone singularity has isotropies bounded by $N$
if for every point of the cone singularity the corresponding isotropy group is either $k^*$ or a finite group of order less than or equal to $N$.
By definition, the only point at which the isotropy group is $k^*$ is the vertex for the torus action.

In order to obtain bounded families of Fano varieties~\cite[Theorem 1.1]{Bir16b},
it is necessary to impose that such projective varieties have mild singularities~\cite{Bir17}.
It does not suffice to assume that the Fano varieties are Kawamata log terminal; 
it is indeed necessary to bound the log discrepancies away from zero.
This forces us to work with the class of $\epsilon$-log canonical singularities for some positive real number $\epsilon$.
Analogously, in order to show that a class of cone singularities is bounded, it is necessary to impose that they have $\epsilon$-log canonical singularities around the vertex.
We prove that a bound on the dimension, singularities, and isotropies are necessary and sufficient to obtain a bounded class of cone singularities:

\begin{introthm}\label{boundedness-cone-singularities}
Let $d$ and $N$ be positive integers and $\epsilon$ a positive real number.
The class of $d$-dimensional $\epsilon$-log canonical cone singularities with isotropies bounded by $N$ forms a bounded family.
\end{introthm}

In subsection~\ref{subsection: examples}, we will give examples where the statement of the theorem fails
if we weaken the conditions on $d,N$ or $\epsilon$.
It is then expected that many algebraic invariants take finitely many possible values on the class of singularities considered in the main theorem.
We remark two invariants which may be of particular interest.

The minimal log discrepancy is an invariant defined to measure the singularities of an algebraic variety (see, e.g.~\cite{Amb99,Mor18}).
Its importance relies on its connection with termination of flips~\cite{Sho04}.
In this direction, we prove the finiteness of minimal log discrepancies of the aforementioned cone singularities.

\begin{introcor}\label{mld}
Let $d$ and $N$ be positive integers and $\epsilon$ a positive real number.
The set of minimal log discrepancies of $d$-dimensional $\epsilon$-log canonical cone singularities with isotropies bounded by $N$ forms a finite set.
\end{introcor}

On the other hand, Chenyang Xu proved the finiteness of the algebraic fundamental group of a klt singularity~\cite{Xu14}.
This result is related to the finiteness of the fundamental group of the smooth locus of log Fano pairs~\cite{TX17}.
In this direction, we prove the existence of a bound for the order of such groups for certain cone singularities.

\begin{introcor}\label{alg fun grp}
Let $d$ and $N$ be positive integers and $\epsilon$ a positive real number.
The possible orders of the algebraic fundamental group of $d$-dimensional $\epsilon$-log canonical cone singularities with isotropies bounded by $N$ form a finite set.
\end{introcor}

\subsection*{Acknowledgements}
The author would like to thank 
Antonio Laface, Alvaro Liendo, Christopher Hacon, Hendrik S{\"u}\ss\,, Jihao Liu, Jingjun Han, and Karl Schwede for many useful comments.
The author was partially supported by NSF research grants no: DMS-1300750, DMS-1265285 and by a
grant from the Simons Foundation; Award Number: 256202.

\section{Preliminaries}

All varieties in this paper are quasi-projective and normal over a fixed algebraically closed field $k$ of characeristic zero unless stated otherwise.
In this section we collect some definitions and preliminary results which will be used in the proof of the main theorem.

\subsection{Cone singularities}

In this subsection, we will introduce the definition of cone singularities which will be used in this paper 
and we will prove some basic properties of these singularities.

\begin{definition}\label{definition:cone-singularity}{\em 
A point $x\in X$ is said to be a {\em cone singularity} if $X$ is a normal affine variety which 
admits an effective $k^*$-action such that $x\in X$ is the unique fixed point which is 
contained in the closure of any orbit for the action.
In the above situation, we say that $k^*$ gives $X$ the structure of a {\em cone singularity}.
The point $x\in X$ which is invariant under the $k^*$-action is called the {\em vertex}.
We will often say that $x\in X$ is a cone singularity to precise that $x$ is the vertex for the $k^*$-action.
}
\end{definition}

\begin{definition}
{\em
Let $X$ be a quasi-projective variety with a $k^*$-action embedded in a projective space $\mathbb{P}^N$.
There exists an open set of the variety $X$ on which all the orbits of the $k^*$-action have the same dimension $d$ and the same degree $k$.
The {\em Chow quotient} of $X$ is the closure of the set of points which correspond to such orbits on ${\rm Chow}_{d,k}(\mathbb{P}^N)$,
the Chow variety which parametrizes cycles of dimension $d$ and degree $k$ on the projective space $\mathbb{P}^N$.
The isomorphism class of the Chow quotient is independent from the embedding of $X$ in a projective space.
The {\em normalized Chow quotient} of $X$ is the normalization of the Chow quotient of $X$.
}
\end{definition}

We start recalling a classic theorem which characterizes affine normal varieties with effective
$k^*$-actions due to Demazure (see, e.g.~\cite[3.5]{Dem88}).

\begin{theorem}\label{theorem:demazure}
Let $X$ be a normal affine variety with an effective $k^*$-action.
Then, we can write 
\[
X \simeq {\rm Spec}\left(
\bigoplus_{n \geq 0} H^0(Y, \mathcal{O}_Y(nD))
\right)
\]
where $Y$ is a quasi-projective variety
and $D$ is a $\qq$-Cartier $\qq$-divisor on $Y$.
\end{theorem}

\begin{notation}\label{notation widetilde}{\em 
Given a normal affine variety $X$ with an effective $k^*$-action, 
we will denote by $Y$ a quasi-projective variety and by $D$ a $\qq$-Cartier $\qq$-divisor on $Y$
realizing the isomorphism in Theorem~\ref{theorem:demazure}.
We will write 
\[
D = \sum_{i=1}^k \frac{p_i}{q_i} D_i,
\]
where the $D_i$'s are pairwise different prime divisors on $Y$ and $p_i$ and $q_i$ are
coprime integers.
We denote by $\widetilde{X}$ the relative spectrum of the divisorial sheaf
\[
\mathcal{A}(D):=\bigoplus_{n \geq 0} \mathcal{O}_Y(nD)
\]
on $Y$. Observe that the natural inclusion of sheaves $\mathcal{O}_Y\hookrightarrow \mathcal{A}(D)$ induces
a good quotient $\pi \colon \widetilde{X}\rightarrow Y$ for the $k^*$-action on $\widetilde{X}$.
The $k^*$-action on $\widetilde{X}$ is induced by the $\nn$-grading of $\mathcal{A}(D)$.
We have a birational contraction $r\colon \widetilde{X}\rightarrow X$ which contracts a divisor $E_0$ on 
$\widetilde{X}$, moreover this divisor is fixed by the $k^*$-action and dominates $Y$.
Hence, we have an induced rational map $X \dashrightarrow Y$ which by abuse of notation
we may denote by $\pi$ as well.
}
\end{notation}

The following proposition gives a characterization of cone singularities,
it follows from~\cite[\S 4]{LS13}.

\begin{proposition}\label{proposition:cone-singularity}
A normal affine variety with an effective $k^*$-action is a cone singularity
if and only if $Y$ is projective and $D$ is a semiample and big $\qq$-Cartier $\qq$-divisor.
In particular, the birational morphism $r\colon \widetilde{X}\rightarrow X$ contracts the divisor $E_0$ to the vertex, 
and we have a good quotient $\pi \colon X-\{ x\} \rightarrow Y$.
\end{proposition}

\begin{remark}\label{minimal}{\em
We know that every cone singularity corresponds to a semiample and big $\qq$-Cartier $\qq$-divisor
on a projective variety $Y$. Furthermore, we may replace the variety $Y$ with the image of the morphism induced by a sufficiently large
and divisible multiple of $D$ to assume that $Y$ is projective and $D$ is an ample $\qq$-Cartier $\qq$-divisor.
This latter model is said to be minimal in the sense of~\cite[Definition 8.7]{AH06}.
Indeed, this variety is the normalized Chow quotient of $Y$, which in this case coincides with the GIT quotient 
since the GIT decomposition has a unique maximal dimensional chamber (see, e.g.~\cite{HK00}).
Observe that pulling-back $D$ to a higher model of $Y$ does not change the cone singularity $x\in X$; 
however, it changes the model $\widetilde{X}$ introduced in~\ref{notation widetilde}.

In what follows, we may say that the cone singularity $x\in X$ corresponds to the 
ample $\qq$-Cartier $\qq$-divisor $D$ on the normalized Chow quotient $Y$, or simply,
corresponds to the couple $(Y,D)$. Observe that our definition of couples differs
from the classic one in which $D$ is assumed to be reduced (see Definition~\ref{pairs}).}
\end{remark}

The following lemma gives a description of the canonical divisor of 
$X$ in terms of the couple $(Y,D)$ (see, e.g.~\cite[Theorem 3.21]{PS11} or~\cite[Theorem 2.8]{Wat81}).

\begin{lemma}\label{lemma:canonical}
Let $x\in X$ be a cone singularity corresponding to the couple $(Y,D)$.
Then the canonical divisor of $X$ is given by
\[
K_X = \pi^*(K_Y)+ \sum_{i=1}^k (q_i-1) \pi_{*}^{-1}D_i,
\]
and the canonical divisor of $\widetilde{X}$ is given by
\[
K_{\widetilde{X}} = \pi^*(K_Y) + \sum_{i=1}^k (q_i-1)\pi_{*}^{-1}D_i - E_0.
\]
\end{lemma}

The following lemma is proved in~\cite[Proposition 1.3.5.7]{ADHL15} 
in the context of Cox rings.

\begin{lemma}\label{lemma:isotropy}
Let $x\in X$ be a cone singularity corresponding to the couple $(Y,D)$. Let $x_0\in X$ any point which is not the vertex.
The order of the isotropy group of $k^*$ at $x_0$ equals the
Cartier index of $D$ at $\pi(x_0)$.
\end{lemma}

\subsection{Log discrepancies}

In this subsection, we will introduce the definition of log discrepancies.
We will prove a formula relating the log discrepancies of a cone singularity 
with the log discrepancies with respect to certain pair structure on the normalized Chow quotient.
This formula is implicit in the proof of~\cite[Theorem 4.7]{LS13}.

\begin{definition}\label{pairs}
{\em 
A {\em couple} $(Y,B)$ consists of a normal quasi-projective algebraic variety $Y$
and a $\qq$-divisor $B$ on $Y$.
A couple $(Y,B)$ is said to be a {\em sub-pair} if the $\qq$-divisor $K_Y+D$ is $\qq$-Cartier.
A sub-pair $(Y,B)$ is said to be a {\em pair} if $D$ is an effective $\qq$-divisor. 
}
\end{definition}

\begin{definition}\label{log discrepancies}
{\em 
Consider a pair $(Y,B)$, a projective birational morphism $f\colon Y'\rightarrow Y$
from a quasi-projective normal variety $Y'$ and a prime divisor $E$ on $Y'$.
We define the {\em log discrepancy} of $(Y,B)$ with respect to $E$ to be 
\[
a_E(K_Y+B) = 1 + {\rm coeff}_E( K_{Y'} - f^*(K_Y+B)).
\]
We say that a pair $(Y,B)$ is {\em $\epsilon$-log canonical} if the log discrepancies
with respect to any prime divisor over $E$ are greater or equal to $\epsilon$.
We say that a pair $(Y,B)$ is {\em Kawamata log terminal}, or simply {\em klt}, 
if the log discrepancies with respect to any prime divisor over $Y$ are greater than zero.
}
\end{definition}

\begin{remark}{\em
A pair $(Y,B)$ is $\epsilon$-log canonical if and only if there exists a projective birational morphism
$f\colon Y'\rightarrow Y$ from a smooth quasi-projective variety $Y'$
such that the exceptional locus $E$ of $f$ is divisorial, the divisor $E\cup f_*^{-1}(B)$
has simple normal crossing support on $Y'$, and the $\qq$-divisor
\[
K_{Y'} - f^*(K_Y+B),
\]
has coefficients greater or equal than $\epsilon-1$.
}
\end{remark}

\begin{proposition}
Let $x\in X$ be a cone singularity corresponding to the couple $(Y,D)$. 
The fields of fractions of $X$ is isomorphic to $k(Y)[M]$ where $M$ is the lattice of characters of the torus.
Hence, every divisorial valuation on $Y$ induces a divisorial valuation on $X$.
\end{proposition}

\begin{notation}
{\em
Given a divisorial valuation $E$ on $Y$ we will denote by $E_X$ the corresponding divisorial valuation on $X$.
Moreover, we will denote by $\chi^u$ the {\em character} of the torus corresponding to $u\in M$.
Hence, every rational function on $X$ has the form $f\chi^u$ where $f$ is a rational function on $Y$ and $u\in M$.
}
\end{notation}

\begin{definition}{\em
Given a projective variety $Y$, a $\qq$-Cartier $\qq$-divisor $D$ on $Y$,
a projective birational morphism $f\colon Y'\rightarrow Y$, 
and a prime divisor $E$ on $Y$, 
we define the {\em Weil index} of $D$ at $E$ to be the smallest positive
integer $\mu$ such that $\mu f^*(D)$ is a Weil divisor at the generic point of $E$,
i.e. the coefficient of $\mu f^*(D)$ at $E$ is an integer.
If $E$ is non-exceptional over $Y$ then the Weil index of $D$ with respect to $E$
is just $q_E$ where $\frac{p_E}{q_E}$ is the coefficient of $D$ at $E$ with $p_E$ and $q_E$ coprime integers.
Observe that the Weil index of $D$ at $E$ does not depend on $f\colon Y'\rightarrow Y$,
it only depends on the divisorial valuation corresponding to $E$.
We will denote the Weil index of $D$ at $E$ by $W_{E}(D)$.
}
\end{definition}

The following proposition is straightforward from the definition of Weil index.

\begin{proposition}\label{prop:weil index vs cartier index}
Let $D$ be a $\qq$-Cartier $\qq$-divisor on a projective variety $Y$.
The Weil index at any exceptional divisor over $Y$ is less than or equal to the Cartier index of $D$.
\end{proposition}

The following proposition is well-known (see e.g.~\cite[Proposition 3.11]{PS11}).

\begin{proposition}\label{prop:cartier is principal}
 Any $k^*$-invariant Cartier divisor is principal on a cone singularity $x\in X$, i.e.
for any $k^*$-invariant Cartier divisor $D$ on $X$ we may find a rational function $f\in k(Y)$ and $u\in M$
such that $D={\rm div}_X(f\chi^u)$.
\end{proposition}

\begin{remark}{\em 
The ring associated to the cone singularity $x\in X$ has a natural $M$-grading (see, e.g.~\cite{AH06}).
However, the weighted monoid of such grading is isomorphic to $\nn$.
We say that an element $u\in M$ is {\em positive} if it lies in the weighted monoid,
and we say it is {\em negative} if its additive inverse is positive.
}
\end{remark}

Since any $k^*$-invariant Cartier divisor $D$ on a cone singularity is principal, we are interested on 
the principal divisor corresponding to the rational function $f\chi^u$ on $x\in X$.
The following proposition gives us the corresponding principal divisor on $X$ (see, e.g.~\cite[Proposition 3.14]{PS11}).

\begin{proposition}\label{prop:div(fchiu)}
Let $x\in X$ be a cone singularity corresponding to the couple $(Y,D)$ and $f\chi^u$ a rational function on $X$.
Then we can write
\[
{\rm div}_{X}(f\chi^u)=  \sum_{i=1}^k q_i\left( u\frac{p_i}{q_i} + {\rm ord}_{D_i}(f) \right) \pi^{-1}_*D_i,
\]
and
\[
{\rm div}_{\widetilde{X}}(f\chi^u) = uE_0+ \sum_{i=1}^k q_i\left( u\frac{p_i}{q_i} + {\rm ord}_{D_i}(f) \right) \pi^{-1}_*D_i.
\]
Here ${\rm div}_X(f)$ denotes the principal divisor on $X$ associated to the rational function $f\in k(X)$.
\end{proposition}

\begin{proposition}\label{prop:comparison of ld}
Let $x\in X$ be a Kawamata log terminal cone singularity corresponding to the ample $\qq$-Cartier $\qq$-divisor $D$ on the projective variety $Y$.
There exist a boundary divisor $B$ on $Y$ such that for each divisorial valuation $E$ over $Y$ we have
\[
a_{E_X}(K_X) = W_{E}(D) a_{E}(K_Y+B).
\]
Moreover, the divisor $-(K_Y+B)$ is an ample $\qq$-Cartier $\qq$-divisor.
\end{proposition}

\begin{proof}
Since the cone singularity $x\in X$ is Kawamata log terminal then $K_X$ is $\qq$-Cartier
and $k^*$-invariant.
Hence, by Proposition~\ref{prop:cartier is principal}, we know that we can write
\begin{equation}\label{eqforkx}
mK_X = {\rm div}_{X}(f \chi^u),
\end{equation}
where $f$ is a rational function on $Y$ and $u \in M$. 
Pushing-forward the divisor $mK_X$ to $Y$ via $\pi$, 
and considering equation~\eqref{eqforkx},
Lemma~\ref{lemma:canonical}, and Proposition~\ref{prop:div(fchiu)}, we obtain
\begin{equation}\label{relonY}
m(K_Y+B) = {\rm div}_Y(f) + uD,
\end{equation}
where $B=\sum_{i=1}^k \frac{q_i-1}{q_i} D_i$ is an effective $\qq$-divisor.
Moreover, since $H:={\rm div}_Y(f)$ is principal we deduce that $(Y,B)$ is a pair
and $-(K_Y+B)$ is an ample $\qq$-divisor since
$u$ and $m$ are integers of opposite sign. 
Hence, it suffices to prove the equality relating the log discrepancies of $(Y,B)$
with those of $X$.
Let $f\colon Y'\rightarrow Y$ a projective birational morphism
from a normal projective variety $Y'$ and $E$ a prime divisor on $Y'$.
Observe that the cone singularity corresponding to the couple $(Y,D)$
is equal to the cone singularity corresponding to the couple $(Y', f^*(D))$ (see Remark~\ref{minimal}).
Let $\widetilde{X}'$ be the relative spectrum on $Y'$ of the divisorial sheaf
\[
\bigoplus_{m\geq 0} \mathcal{O}_{Y'}(m f^*(D)).
\]
Observe that the center of $E_X$ on $\widetilde{X}'$ is just the 
strict transform of $E$ on $\widetilde{X}'$.
Thus, we have a commutative diagram as follows:
\[
 \xymatrix{
  \widetilde{X}' \ar[r]^-{\widetilde{f}}\ar[d]_-{\pi'}   & \widetilde{X}\ar[r]^-{r} \ar[d]_-{\pi} & X \ar@{-->}[ld]^-{\pi} \\
  Y' \ar[r]^-{f} & Y & 
 }
\]
By Proposition~\ref{prop:div(fchiu)}, we have that
\[
{\rm coeff}_{E_X}( \widetilde{f}^*(r^*(K_X))) = \frac{{\rm div}_{\widetilde{X}'}(f\chi^u)}{m} = \frac{W_E(D)}{m} 
\left(
u {\rm coeff}_E (f^*(D)) + {\rm coeff}_E(f^*(H))
\right).
\]
On the other hand, by Lemma~\ref{lemma:canonical} we have that
\[
{\rm coeff}_{E_X}(K_{\widetilde{X}'}) = W_E(D){\rm coeff}_E(K_Y) + W_E(D)-1.
\]
Hence, the log discrepancy of $K_X$ is given by
\[
a_{E_X}(K_X) = W_E(D) \left(
{\rm coeff}_E(K_Y) - \frac{u}{m} {\rm coeff}_E(f^*(D)) - \frac{1}{m}{\rm coeff}_E(f^*(H))+1
\right).
\]
From equation~\eqref{relonY}, we deduce that
\[
\frac{u}{m} {\rm coeff}_E(f^*(D)) - \frac{1}{m}{\rm coeff}_E(f^*(H)) = - {\rm coeff}_E( f^*(K_Y+B)),
\]
so we can write 
\[
a_{E_X}(K_X) =  W_E(D)\left(  {\rm coeff}_E(K_Y) - {\rm coeff}_E(f^*(K_Y+B)) +1 \right) = W_E(D)a_E(Y,D).
\]
Thus, for any divisor $E$ over $Y$ we get the relation
\[
a_{E_X}(K_X) = W_E(D)a_E(K_Y+B).
\]
\end{proof}

\begin{definition}
{\em 
A log pair $(Y,B)$ is said to be {\em log Fano} if it is klt and $-(K_Y+B)$ is an ample $\qq$-Cartier $\qq$-divisor.
In what follows, we may call $(Y,B)$ the {\em log Fano quotient} of the cone singularity.
Observe that the log Fano quotient $(Y,B)$ of a cone singularity $x\in X$ may not be equal
to the corresponding couple $(Y,D)$.
}
\end{definition}

The following remark relates the log Fano quotient and the corresponding couple
of a klt cone singularity.

\begin{remark}\label{rem:structure of B}
{\em 
Let $x\in X$ be a cone singularity corresponding to the ample $\qq$-Cartier $\qq$-divisor $D$ on the projective variety $Y$.
The boundary divisor on $Y$ associated to the log Fano quotient of the cone singularity is
\[
B=\sum_{i=1}^k \left( 1- \frac{1}{q_i}\right) D_i,
\]
where the $q_i$'s are as in Notation~\ref{notation widetilde}.
This means that the log Fano quotient of a cone singularity has a bondary with standard coefficients in the sense of~\cite[Definition 2.4]{Mor18}.
Observe that the log Fano quotient of a cone singularity is uniquely determined by the pair $(Y,D)$.
On the other hand, the isomorphism class of $x\in X$ is determined by the couple $(Y,D)$ up to isomorphisms in the first component
and linear equivalence on the second component, i.e. if $D_1$ and $D_2$ are two $\qq$-Cartier $\qq$-divisors
such that $D_1\sim D_2$, then the couples $(Y,D_1)$ and $(Y,D_2)$ determine isomorphic cone singularities (see, e.g.~\cite[Proposition 8.6]{AH06}).
Here, we say that two $\qq$-divisors $D_1$ and $D_2$ are linearly equivalent if $D_1-D_2$ is a principal divisor on $Y$.
}
\end{remark}

\begin{corollary}\label{prop:e/N-log canonical}
Let $\epsilon$ be a positive real number and $N$ be a positive integer.
Let $x\in X$ be an $\epsilon$-log canonical cone singularity with isotropies bounded by $N$, 
then its log Fano quotient $(Y,B)$ is $\frac{\epsilon}{N}$-log canonical.
\end{corollary}

\begin{proof}
Let $(Y,D)$ the couple corresponding to the cone singularity $x\in X$,
and let $E$ be a prime divisor over $Y$.
By Proposition~\ref{prop:weil index vs cartier index},
we know that the Weil index of $D$ at $E$ is at most the Cartier index of $D$.
On the other hand, by Proposition~\ref{lemma:isotropy}, we know that the Cartier index of $D$ is bounded by $N$.
Therefore, we have the inequality $W_E(D)\leq N$ for any prime divisor $E$ over $Y$.
By Proposition~\ref{prop:comparison of ld}, we obtain
\[
a_E(K_Y+B) = \frac{ a_{E_X}(K_X)}{ W_E(D)} \geq \frac{\epsilon}{N}.
\]
\end{proof}

\subsection{Bounded families of Fano varieties}
In this subsection, we recall the boundedness of Fano varieties due to Birkar~\cite{Bir16a,Bir16b}
and a result about the Neron-Severi space on families of Fano varieties due to Hacon and Xu (see~\cite[Proposition 2.8]{HX15}).

\begin{definition}{\em 
We say that a class of schemes $\mathcal{C}$ is {\em bounded} if there exists a morphism $\phi \colon \mathcal{X} \rightarrow T$
between two schemes of finite type such that every sceheme on the class $\mathcal{C}$ appears as
a geometric fiber of $\phi$.

If the class of schemes $\mathcal{C}$ is a class of projective varieties we will also require that $\phi$ is a projective morphism
between possibly reducible quasi-projective varieties.

If $\mathcal{C}$ is a class of couples $(Y,B)$ we say that it is {\em log bounded} if there 
exists a projective morphism $\phi\colon \mathcal{X}\rightarrow T$ between quasi-projective varieties
and a $\qq$-divisor $\mathcal{B} \subset \mathcal{X}$ such that for each $(Y,B)\in \mathcal{C}$
there exists a closed point $t\in T$ and an isomorphism $Y \simeq \mathcal{X}_t$ so that
the support of $B$ is contained in the support of $\mathcal{B}_t$ under this isomorphism.
Moreover, we say that the class of couples $(Y,B)$ is {\em log bounded with coefficients} or {\em strictly log bounded}
if the isomorphism $Y\simeq \mathcal{X}_t$ induces an isomorphism of $B$ and $\mathcal{B}_t$ with their corresponding coefficients.	
We will call $\phi\colon \mathcal{X}\rightarrow T$ the {\em bounding family} and $\mathcal{B}\subset \mathcal{X}$ the {\em bounding divisor}.

We say that a class $\mathcal{C}$ of couples $(Y,D)$ is {\em log bounded with coefficients up to linear equivalence}
or {\em strictly log bounded up to linear equivalence},
if there exists a class of couples $\mathcal{C}'$ which is log bounded with coefficients, such that for
each $(Y,D)\in \mathcal{C}$ there exists $(Y',D')\in \mathcal{C}'$ and an isomorphism $f\colon Y\rightarrow Y'$
for which $f^*(D')\sim D$, or equivalently, $D' \sim f_*(D)$.

Finally, we say that a class of singularities $\mathcal{C}$ is {\em bounded}, if for each $x\in X$ belonging to $\mathcal{C}$ 
we can find an affine neighborhood of $x\in U\subset X$, so that the schemes $U$ are bounded in the above sense.
}
\end{definition}

The following proposition is a consequence of the functoriality of polyhedral divisors (see, e.g.~\cite[Proposition 8.6]{AH06}).

\begin{proposition}\label{boundedness of cones}
A class of cone singularities is bounded if the class of corresponding couples 
is strictly log bounded up to linear equivalence.
\end{proposition}

\begin{proof}
Let $\mathcal{C}_X$ be a class of cone singularities and $\mathcal{C}_Y$ be the corresponding class of couples.
By~\cite[Theorem 3.10]{Ale94}, it suffices to prove that for every sequence $x_i\in X_i$ of cone singularities there exists an infinite sub-sequence which is bounded.
Let $(Y_i,D_i)$ be the corresponding sequence of couples, where $Y_i$ is projective and $D_i$ is an ample $\qq$-Cartier $\qq$-divisor on $Y_i$.
Let $\phi \colon \mathcal{X}\rightarrow T$ and $\mathcal{D}\subset \mathcal{X}$ be a family and a divisor realizing the strictly log boundedness of $(Y_i,D_i)$ 
up to linear equivalence.
Passing to a subsequence, we may assume that the points $t_i$ corresponding to the couples $(Y_i,D_i)$ are dense on $T$.
Now, we consider the variety 
\[
\mathcal{X}_X :={\rm Spec} \left( \bigoplus_{m\geq 0} H^0(\mathcal{X}/T, m \mathcal{D})\right),
\]
which has an structure morphism $\phi_X \colon \mathcal{X}_X \rightarrow T$.
By construction we have isomorphisms
\[
\mathcal{X}_{X,t_i} \simeq {\rm Spec} \left( \bigoplus_{m\geq 0}H^0(\mathcal{X}_{t_i}, m\mathcal{D}_{t_i} ) \right)
\simeq 
{\rm Spec} \left( \bigoplus_{m\geq 0}H^0(Y_i, mD_i) \right)
\simeq 
X_t.
\]
Therefore, we conclude that the morphism $\phi_X \colon \mathcal{X}_X \rightarrow T$ is a bounding family for the cone singularities $x_i\in X_i$ 
which belong to the class $\mathcal{C}_X$.
\end{proof}

\begin{theorem}\label{theorem:bab}
Let $d$ be a positive integer and $\epsilon$ a positive real number.
The set of varieties $Y$ for which the following conditions hold:
\begin{itemize}
\item $Y$ is a projective variety of dimension $d$,
\item There exists a boundary divisor $B$ on $Y$ such that $(Y,B)$ is an $\epsilon$-log canonical pair, and
\item the $\qq$-Cartier $\qq$-divisor $-(K_Y+B)$ is ample,
\end{itemize}
forms a bounded family.
\end{theorem}

We are also interested to bound the pairs $(Y,B)$.
In order to do so, we also need to impose a condition on the coefficients of $B$ as the following corollary shows.

\begin{corollary}\label{log boundedness}
Let $d$ be a positive integer, $\epsilon$ a positive real number and $\mathcal{R}$ a set of rational numbers satisfying the descending chain condition.
The set of pairs $(Y,B)$ for which the following conditions hold:
\begin{itemize}
\item $Y$ is a projective variety of dimendion $d$,
\item the pair $(Y,B)$ is $\epsilon$-log canonical,
\item the $\qq$-Cartier $\qq$-divisor $-(K_Y+B)$ is ample, and
\item the coefficients of $B$ belong to $\mathcal{R}$.
\end{itemize}
forms a log bounded family.
Moreover, if $\mathcal{R}$ is finite, then the pairs $(Y,B)$ forms a strictly log bounded family.
\end{corollary}

\begin{proof}
By Theorem~\ref{theorem:bab}, we may find a positive real number $C$ and for each $Y$ as in the statement we may choose a very ample ample Cartier divisor $A_Y$ on $Y$
so that $A_Y^d \leq C$. We may further assume that $(-K_Y) \cdot A^{d-1} \leq C$.
Let $\delta$ be a positive rational number which is smaller than any element of $\mathcal{R}$, 
then we have that 
\[
\frac{1}{\delta} {\rm red}(B) \cdot A^{d-1} \leq B \cdot A^{d-1} \leq -K_Y \cdot A^{d-1} \leq C.
\]
Thus, by~\cite[Lemma 3.7.(2)]{Ale94} we conclude that the pairs $(Y,B)$ are log bounded.
If $\mathcal{R}$ is finite, then the log boundedness of $(Y,B)$ with coefficients follows from the 
log boundedness of $(Y,B)$ by taking all the possible combinations for the coefficients of the bounding divisor on $\mathcal{R}$.
\end{proof}

\begin{definition}{\em

Given a projective morphism $\phi \colon \mathcal{X} \rightarrow T$ we say that $\mathcal{X}$ is of {\em Fano type over $T$}
if there exists a big boundary $\mathcal{B}$ over $T$ on $\mathcal{X}$ such that $(\mathcal{X},\mathcal{B})$ is klt and $K_{\mathcal{X}}+\mathcal{B}\sim_{\qq,T} 0$.
}
\end{definition}

\begin{proposition}\label{finiteness of cox rings}
Let $\phi  \colon \mathcal{X}\rightarrow T$ be a projective morphism such that $\mathcal{X}$ is of Fano type over $T$.
Up to a base change, for every $t\in T$ the following three conditions hold:
\begin{itemize}
\item The restriction morphism $\rho_t \colon N^1(\mathcal{X}/T)\rightarrow N^1(\mathcal{X}_t)$ is an isomorphism, 
\item the restriction morphism induce isomorphisms ${\rm Cox}(\mathcal{X}/T)\simeq {\rm Cox}(\mathcal{X}_t)$,
\item we have that $\rho_t({\rm Mov}(\mathcal{X}/T)) ={\rm Mov}(\mathcal{X}_t)$, and
\item there is a one-to-one correspondence between the two Mori chamber decompositions.
\end{itemize}
\end{proposition}

\begin{proof}
The first, third and fourth claims are proved in~\cite[Proposition 2.8]{HX15}.
Since the Cox ring ${\rm Cox}(\mathcal{X}/T)$ is finitely generated (see, e.g.~\cite[Corollary 1.3.2]{BCHM10}),
after shrinking $T$ we may assume that the restriction of the generating sections of the Cox ring ${\rm Cox}(\mathcal{X}/T)$ to $\mathcal{X}_t$ are injective for every $t\in T$.
On the other hand, by~\cite[Proposition 2.7]{HX15} or~\cite[Theorem 1.1]{dFH11}, we know that they are also surjective, proving the second claim.
\end{proof}

\subsection{Examples}\label{subsection: examples}

In this subsection, we give some examples in which the statement of the main theorem
does not hold if we weaken the assumptions on $d,\epsilon$ or $N$.

\begin{example}\label{smooth point}
{\em
Let $(x_0,\dots,x_d)$ be the coordinates of the $(d+1)$-dimensional affine space $k^{d+1}$.
Consider the cone singularity structure on $k^{d+1}$ given by the diagonal action of $k^*$.
In this case, the action has trivial isotropies, e.g. the isotropy at every closed point which is not the origin is trivial
while the isotropy at the origin is $k^*$. Moreover, smooth points are $1$-log canonical and form an unbounded family 
whenever the dimension is not bounded. Hence, the statement of the main theorem
fails if we drop the condition on the dimension.
}
\end{example}

\begin{example}\label{cone over rational curves}{\em 
In this example, we show that the statement of the main theorem fails if we drop the condition on $\epsilon$-log canonical singularities.
Indeed, consider the cone over a rational curve of degree $m$, i.e. the cone singularity given by
\[
X_m := {\rm Spec}\left(
\bigoplus_{n \geq 0 } H^0(\pp^1,\mathcal{O}_{\pp^1}(nmH))
\right)
\]
where $H$ is the class of a point on $\pp^1$.
It is well-known that the log discrepancy at the exceptional divisor obtained by blowing-up the maximal ideal of the vertex of $X_m$ is $\frac{2}{n}$.
Hence, the cone singularities $X_m$ are not $\epsilon$-log canonical for some fixed positive real number $\epsilon$, even if all of them are log canonical.
However, the $k^*$-action given by the grading has trivial isotropies and all these are $2$-dimensional singularities.
This sequence of surfaces singularities give an example in which the main theorem fails if we drop the condition on $\epsilon$.
Indeed, the Cartier index of $K_{X_m}$ is $m$.
Thus, for $m$ unbounded the above sequence of $2$-dimensional log canonical singularities with isotropies bounded by $1$ does not form a bounded family.
}
\end{example}

\begin{example}\label{A_n singularities}
{\em 
In this example, we show that the statement of the main theorem fails if we drop the condition on the isotropies bounded by $N$.
Consider the well-known $A_n$-singularities:
\[
A_n:=\{ (x,y,z) \mid x^2y^2-z^n =0 \} \subset k^3.
\]
We claim that any $k^*$-action on $A_n$ which give it the structure of a cone singularity 
has isotropy greater or equal than $n$ along either the curve $x=0$ or the curve $y=0$.
Indeed, any $k^*$-action on $A_n$ is induced  by a sub-torus of the torus action 
$(k^*)^2$ on $A_n$ given by
\[
(t_1,t_2) \cdot (x,y,z) = (t_1 t_2^n x, t_1^{-1}t_2^n y, t_2^2z).
\]
A subtorus embedding $k^*\hookrightarrow (k^*)^2$ has the form $t\mapsto (t^{a},t^{b})$
for certain integers $a$ and $b$.
Hence, any $k^*$-action on $A_n$ is given by
\begin{equation}\label{action on A_n}
t \cdot (x,y,z) = (t^{a+bn}x, t^{-a+bn}y, t^{2b} z),
\end{equation}
where $a$ and $b$ are integers.
We check that the above action gives $A_n$ a cone singularity structure if and only if $b\neq 0$.
If $b=0$ the action~\eqref{action on A_n} is given by $t\cdot (x,y,z)=(t^ax, t^{-a}y,z)$
and the curve $z=1$ is an orbit which does not contain the origin in its closure.
Hence, we may assume that $b\neq 0$ and therefore 
the action~\eqref{action on A_n} has isotropy 
$-a+bn$ on the curve $x=0$ and isotropy $a+bn$ on the curve $y=0$.
Observe that we have
\[
2|b|n = |2bn| = | (a+bn) + (-a+bn)| \leq |a+bn| + |-a+bn|.
\]
Therefore, for $n$ large enough, either the isotropy at $x=0$ or the isotropy at $y=0$
is getting arbitrarily large.
Observe that the $A_n$ singularities are canonical surface singularities which don't form a bounded family since their algebraic fundamental groups have arbitrarily large order.
}
\end{example}

\section{Proof of boundedness}

\begin{proof}[Proof of Theorem~\ref{boundedness-cone-singularities}]
Let $d$ and $N$ be positive integers and $\epsilon$ a positive real number.
Denote by $\mathcal{C}_{d,\epsilon,N}$ the class of $d$-dimensional $\epsilon$-log canonical cone singularities with isotorpies bounded by $N$.
Denote by $\mathcal{C}^{\rm quot}_{d,\epsilon,N}$ the class of corresponding couples $(Y,D)$ associated to the cone singularities in $\mathcal{C}_{d,\epsilon,N}$.
By Proposition~\ref{boundedness of cones} and~\cite[Theorem 3.1]{Ale94}, it suffices to prove that for every sequence $(Y_i,D_i) \in \mathcal{C}^{\rm quot}_{d,\epsilon,N}$
we can find an infinite sub-sequence which is log bounded with coefficients up to linear equivalence. We denote by $x_i\in X_i$ the corresponding sequence of cone singularities.\\

\textbf{Step 1:} In this step, we prove that the log Fano quotients $(Y_i,B_i)$ of $x_i \in X_i$ belong to a strictly log bounded family which only depends on $d,\epsilon$ and $N$. 
Indeed, by Lemma~\ref{prop:e/N-log canonical}, we know that the log Fano quotient $(Y_i,B_i)$ has $\frac{\epsilon}{N}$-log canonical singularities
and $-(K_{Y_i}+B_i)$ is an ample $\qq$-Cartier $\qq$-divisor.
Moreover, since the coefficients of $B_i$ have the form $1-\frac{1}{n}$ for some positive integer $n$ and the pairs $(Y_i,B_i)$ are $\frac{\epsilon}{N}$-log canonical,
we conclude that $n$ is at most $\frac{N}{\epsilon}$.
Thus, the coefficients of $B_i$ belong to a finite set which only depends on $\epsilon$ and $N$.
Hence, we can apply Corollary~\ref{log boundedness} to deduce that the pairs $(Y_i,B_i)$ belong to a strictly log bounded family which only depends on $d-1, \epsilon$ and $N$.
We denote by $\phi \colon \mathcal{X}\rightarrow T$ the bounding family for the $Y_i$'s and $\mathcal{B}\subset \mathcal{X}$ the bounding divisor for the $B_i$'s.
We denote by $\mathcal{D}_i$ the $\qq$-Cartier $\qq$-divisor on $\mathcal{X}$ such that $(\mathcal{D}_i)|_{t_i} = D_i$.
As in the proof of~\cite[Proposition 2.8]{HX15}, we may assume that all fibers of $\mathcal{X}\rightarrow T$ are $\qq$-factorial.
Observe that this last assumption does not change the isomorphism class of the cone singularities, however it may change the models $\widetilde{X}_i$.\\

\textbf{Step 2:} We denote by $t_i \in T$ the points on the base of the bounding family for which
$(\mathcal{X}_{t_i}, \mathcal{B}_{t_i})\simeq (Y_i,B_i)$ holds.
Up to a base change and shrinking $T$ we may assume that the $t_i$'s are dense on $T$ 
and that ${\rm Cox}(\mathcal{X}/T)\simeq {\rm Cox}(\mathcal{X}_{t})$ for every $t\in T$ (see Proposition~\ref{finiteness of cox rings}).
Shrinking more if needed, we may assume that the points of $N^1(\mathcal{X}/T)$ corresponding to classes of Weil divisors (resp. Cartier divisors)
are identified via $\rho_t$ with the points of $N^1(\mathcal{X}_t)$ corresponding to classes of Weil divisors (resp. Cartier divisors).\\

\textbf{Step 3:}  In this step we will compute the log discrepancy of the divisor $E_i$ contracted by $\widetilde{X}_i \rightarrow X_i$
with respect to the canonical divisor $K_{X_i}$. This computation will be expressed in terms of the log Fano quotient $(Y_i,B_i)$ of the cone singularity.
By the proof of Proposition~\ref{prop:comparison of ld}, we know that we can write
\[
D_i = \frac{m_i}{u_i} (K_{Y_i}+B_i) + \frac{1}{u_i} H_i,
\]
where $m_i$ and $u_i$ are integers of opposite sign,
and $H_i$ is a principal divisor on $Y_i$.
By Lemma~\ref{lemma:canonical} and Proposition~\ref{prop:div(fchiu)}, we conclude that there is an equality
\[
a_{E_i}(K_{X_i}) = \frac{u_i}{m_i}.
\]
Hence, by the assumption on $X_i$ being $\epsilon$-log canonical we deduce that $\frac{m_i}{u_i}\leq \frac{1}{\epsilon}$.\\

\textbf{Step 4:} In this step we prove that the possible classes of $\qq$-linearly equivalence of $\mathcal{D}_i$ on $N^1(\mathcal{X}/T)$ belong to a finite set.
Observe that we have the relation 
\begin{equation}\label{eqqlin}
\mathcal{D}_i \sim_{\qq,T} \frac{m_i}{u_i} (K_{\mathcal{X}/T}+\mathcal{B}_i),
\end{equation}
for each $i$.
Observe that the set 
\[
\left\{ \mathcal{F} \in N^1(\mathcal{X}/T) \mid \mathcal{F} \sim_{\qq,T} r (K_{\mathcal{X}/T}+\mathcal{B}_i) \text{ and } 0\leq r \leq \frac{1}{\epsilon}\right\}
\]
 is a compact subspace of $N^1(\mathcal{X}/T)$.
On the other hand, since $N\mathcal{D}_i$ is a Cartier divisor, the divisors $\mathcal{D}_i$ belong to a lattice inside $N^1(\mathcal{X}/T)$.
From the equation~\eqref{eqqlin} we conclude that there are finitely many possible $\qq$-linearly classes $\mathcal{F}_1,\dots, \mathcal{F}_r$ on $N^1(\mathcal{X}/T)$
for which any $\mathcal{D}_i$ is $\qq$-linearly equivalent to some $\mathcal{F}_j$.
Passing to a subsequence we may assume that all the $\mathcal{D}_i$ are $\qq$-linearly equivalent to each other.\\

\textbf{Step 5:} We prove that the sequence $(\mathcal{X},\mathcal{D}_i)$ is log bounded up to linear equivalence over $T$.
We can write
\[
\mathcal{D}_i = \mathcal{D}_{i,{\rm f}} + \mathcal{D}^{+}_{i,{\rm W}} - \mathcal{D}^{-}_{i, {\rm W}},
\]
where the three divisors are effective, $\mathcal{D}_{i,{\rm f}}$ has coefficients on the interval $(0,1)\cap \zz\left[\frac{1}{N}\right]$,
and the two latter divisors are integral.
Since the pairs $(\mathcal{X},\mathcal{B}_i)$ are strictly log bounded, we conclude that the couples $(\mathcal{X},\mathcal{D}_{i,f})$
are strictly log bounded as well.
Thus, passing to a subsequence we may assume that for each $i$ and $j$ we have that 
\[
\mathcal{D}^{+}_{i,W} - \mathcal{D}^{-}_{i,W} =
\mathcal{D}^{+}_{j,W} - \mathcal{D}^{-}_{j,W}.
\]
Since the Cox ring of $\mathcal{X}$ relative to $T$ is finitely generated, we may find a finite basis 
$E_1,\dots, E_r$ for the effective Weil divisors on $N^1(\mathcal{X}/T)$ up to linear equivalence.
We will denote by $k$ the smallest positive integer such that for every Weil divisor $E$ on $\mathcal{X}$ the multiple $kE$ is Cartier.
For each $i$ we can write
\[
\mathcal{D}_{i,W}^{+} \sim_{T} \sum_{l=1}^r a^+_{i,l} E_l + \sum_{l=1}^r k b^+_{i,l} E_l,
\]
and 
\[
\mathcal{D}_{i,W}^{-} \sim_{T} \sum_{l=1}^r a^{-}_{i,l} E_l + \sum_{l=1}^r k b^{-}_{i,l} E_l,
\]
where the $a^{+}_{i,l}$'s and the $a^{-}_{i,l}$'s are positive integers in the interval $(0,k-1]$,
while the $b^{+}_{i,l}$'s and $b^{-}_{i,l}$'s are positive integers.
Observe that there are finitely many possible Weil divisors
\[
\sum_{l=1}^r a^+_{i,l} E_l  \quad \text{ and } \quad  \sum_{l=1}^r a^{-}_{i,l} E_l,
\]
Hence, passing to a subsequence we may assume that for every $i$ and $j$ we have
\[
\sum_{l=1}^r k b^+_{i,l} E_l - \sum_{l=1}^r k b^{-}_{i,l} E_l \sim_{\qq,T} 
\sum_{l=1}^r k b^+_{j,l} E_l - \sum_{l=1}^r k b^{-}_{j,l} E_l
\]
are two Cartier divisors which are $\qq$-linearly equivalent over $T$, therefore for every $i$ and $j$ we have
\[
\sum_{l=1}^r k b^+_{i,l} E_l - \sum_{l=1}^r k b^{-}_{i,l} E_l \sim_{T} 
\sum_{l=1}^r k b^+_{j,l} E_l - \sum_{l=1}^r k b^{-}_{j,l} E_l.
\]
Thus, for each $i$ we can write
\[
\mathcal{D}_i \sim_{T} \mathcal{D}_{i,{\rm f}} + \left(  \sum_{l=1}^r a^+_{i,l} E_l  - \sum_{l=1}^r a^{-}_{i,l} E_l  \right)
+\left( 
\sum_{l=1}^r k b^+_{i,l} E_l - \sum_{l=1}^r k b^{-}_{i,l} E_l
\right)
\]
where there are finitely many possible $\qq$-divisors 
\[
\mathcal{D}_{i,{\rm f}} + \left(  \sum_{l=1}^r a^+_{i,l} E_l  - \sum_{l=1}^r a^{-}_{i,l} E_l  \right)
\]
and all the Weil divisors 
\[
\sum_{l=1}^r k b^+_{i,l} E_l - \sum_{l=1}^r k b^{-}_{i,l} E_l
\]
are linearly equivalent over $T$ to a fixed integral divisor.
Thus, we deduce that the sequence $(\mathcal{X}, \mathcal{D}_i)$ is strictly log bounded up to linear equivalence over $T$.\\

\textbf{Step 6:} In this step we complete the proof;
we prove that the couples $(Y_i,D_i)$ are strictly log bounded up to linear equivalence.
Indeed, by the fifth step we may pass to a subsequence in which
\[
K_{\mathcal{X}} + \mathcal{D}_i \sim_T K_{\mathcal{X}}+\mathcal{D}_j
\]
for every $i$ and $j$.
Moreover, since the $t_i$'s are dense in $T$ we have that the linear equivalence 
\[
(K_{\mathcal{X}} + \mathcal{D}_i)|_{t_l} \sim (K_{\mathcal{X}}+\mathcal{D}_j)|_{t_l},
\]
holds for all but finitely many $t_l$'s.
Thus, we conclude that for all but finitely many $t_l$'s the linear equivalence
\[
K_{Y_l} + D_l \sim K_{Y_l} + D_{l,1}
\]
holds, where $D_{l,1}=(\mathcal{D}_1)|_{t_l}$.
Since the pairs $(Y_l, D_{l,1})$ are strictly log bounded by
the family $\phi\colon \mathcal{X}\rightarrow T$ and the divisor $\mathcal{D}_1$, 
we conclude that the pairs $(Y_l,D_l)$ are strictly log bounded up to linear equivalence.
\end{proof}

\begin{proof}[Proof of Corollary~\ref{mld}]
This follows from Theorem~\ref{boundedness-cone-singularities}, and the fact that minimal log discrepancies take finitely many values on bounded familes~\cite[\S 2]{Amb99}.
\end{proof}

\begin{proof}[Proof of Corollary~\ref{alg fun grp}]
This follows from Theorem~\ref{boundedness-cone-singularities}, and the upper semi-continuity of the order of algebraic fundamental groups~\cite[Corollary 17]{BKS03}.
\end{proof}


\begin{bibdiv}
\begin{biblist}

%\bib{Ale93}{article}{
%   author={Alexeev, Valery},
%   title={Two two-dimensional terminations},
%   journal={Duke Math. J.},
%   volume={69},
%   date={1993},
%   number={3},
%   pages={527--545},
%   issn={0012-7094},
%   review={\MR{1208810}},
%   doi={10.1215/S0012-7094-93-06922-0},
%}

\bib{Ale94}{article}{
   author={Alexeev, Valery},
   title={Boundedness and $K^2$ for log surfaces},
   journal={Internat. J. Math.},
   volume={5},
   date={1994},
   number={6},
   pages={779--810},
   issn={0129-167X},
   review={\MR{1298994}},
   doi={10.1142/S0129167X94000395},
}
	
\bib{Amb99}{article}{
   author={Ambro, Florin},
   title={On minimal log discrepancies},
   journal={Math. Res. Lett.},
   volume={6},
   date={1999},
   number={5-6},
   pages={573--580},
   issn={1073-2780},
   review={\MR{1739216}},
   doi={10.4310/MRL.1999.v6.n5.a10},
}


\bib{AIPSV12}{article}{
   author={Altmann, Klaus},
   author={Ilten, Nathan Owen},
   author={Petersen, Lars},
   author={S\"{u}\ss , Hendrik},
   author={Vollmert, Robert},
   title={The geometry of $T$-varieties},
   conference={
      title={Contributions to algebraic geometry},
   },
   book={
      series={EMS Ser. Congr. Rep.},
      publisher={Eur. Math. Soc., Z\"{u}rich},
   },
   date={2012},
   pages={17--69},
   review={\MR{2975658}},
   doi={10.4171/114-1/2},
}
	

%
%\bib{Amb11}{article}{
%   author={Ambro, Florin},
%   title={Basic properties of log canonical centers},
%   conference={
%      title={Classification of algebraic varieties},
%   },
%   book={
%      series={EMS Ser. Congr. Rep.},
%      publisher={Eur. Math. Soc., Z\"{u}rich},
%   },
%   date={2011},
%   pages={39--48},
%   review={\MR{2779466}},
%   doi={10.4171/007-1/2},
%}

\bib{ADHL15}{book}{
   author={Arzhantsev, Ivan},
   author={Derenthal, Ulrich},
   author={Hausen, J\"{u}rgen},
   author={Laface, Antonio},
   title={Cox rings},
   series={Cambridge Studies in Advanced Mathematics},
   volume={144},
   publisher={Cambridge University Press, Cambridge},
   date={2015},
   pages={viii+530},
   isbn={978-1-107-02462-5},
   review={\MR{3307753}},
}

\bib{AH06}{article}{
   author={Altmann, Klaus},
   author={Hausen, J\"{u}rgen},
   title={Polyhedral divisors and algebraic torus actions},
   journal={Math. Ann.},
   volume={334},
   date={2006},
   number={3},
   pages={557--607},
   issn={0025-5831},
   review={\MR{2207875}},
   doi={10.1007/s00208-005-0705-8},
}


\bib{AHS08}{article}{
   author={Altmann, Klaus},
   author={Hausen, J\"{u}rgen},
   author={S\"{u}ss, Hendrik},
   title={Gluing affine torus actions via divisorial fans},
   journal={Transform. Groups},
   volume={13},
   date={2008},
   number={2},
   pages={215--242},
   issn={1083-4362},
   review={\MR{2426131}},
   doi={10.1007/s00031-008-9011-3},
}

\bib{Bir16a}{misc}{
  author = {Birkar, Caucher},
  title={Anti-pluricanonical systems on Fano varieties},
  year = {2016},
  note = {https://arxiv.org/abs/1603.05765v3},
}

\bib{Bir16b}{misc}{
  author = {Birkar, Caucher},
  title={Singularities of linear systems and boundedness of Fano varieties},
  year = {2016},
  note = {https://arxiv.org/abs/1609.05543v1},
}

\bib{Bir17}{misc}{
  author = {Birkar, Caucher},
  title={Birational geometry of algebraic varieties},
  year = {2017},
  note = {https://arxiv.org/abs/1801.00013},
}

\bib{Bro13}{article}{
   author={Brown, Morgan},
   title={Singularities of Cox rings of Fano varieties},
   language={English, with English and French summaries},
   journal={J. Math. Pures Appl. (9)},
   volume={99},
   date={2013},
   number={6},
   pages={655--667},
   issn={0021-7824},
   review={\MR{3055212}},
   doi={10.1016/j.matpur.2012.10.003},
}
		

%\bib{Bor97}{article}{
%   author={Borisov, Alexandr},
%   title={Minimal discrepancies of toric singularities},
%   journal={Manuscripta Math.},
%   volume={92},
%   date={1997},
%   number={1},
%   pages={33--45},
%   issn={0025-2611},
%   review={\MR{1427666}},
%   doi={10.1007/BF02678179},
%}

\bib{BCHM10}{article}{
   author={Birkar, Caucher},
   author={Cascini, Paolo},
   author={Hacon, Christopher D.},
   author={McKernan, James},
   title={Existence of minimal models for varieties of log general type},
   journal={J. Amer. Math. Soc.},
   volume={23},
   date={2010},
   number={2},
   pages={405--468},
   issn={0894-0347},
   review={\MR{2601039}},
   doi={10.1090/S0894-0347-09-00649-3},
}

\bib{BKS03}{collection}{
   title={Higher dimensional varieties and rational points},
   series={Bolyai Society Mathematical Studies},
   volume={12},
   editor={B\"{o}r\"{o}czky, K\'{a}roly, Jr.},
   editor={Koll\'{a}r, J\'{a}nos},
   editor={Szamuely, Tam\'{a}s},
   note={Papers from the Summer School and Conference held in Budapest,
   September 3--21, 2001},
   publisher={Springer-Verlag, Berlin; J\'{a}nos Bolyai Mathematical Society,
   Budapest},
   date={2003},
   pages={310},
   isbn={3-540-00820-9},
   review={\MR{2011742}},
   doi={10.1007/978-3-662-05123-8},
}

\bib{CLS11}{book}{
   author={Cox, David A.},
   author={Little, John B.},
   author={Schenck, Henry K.},
   title={Toric varieties},
   series={Graduate Studies in Mathematics},
   volume={124},
   publisher={American Mathematical Society, Providence, RI},
   date={2011},
   pages={xxiv+841},
   isbn={978-0-8218-4819-7},
   review={\MR{2810322}},
   doi={10.1090/gsm/124},
}

	
\bib{Dem88}{article}{
   author={Demazure, Michel},
   title={Anneaux gradu\'{e}s normaux},
   language={French},
   conference={
      title={Introduction \`a la th\'{e}orie des singularit\'{e}s, II},
   },
   book={
      series={Travaux en Cours},
      volume={37},
      publisher={Hermann, Paris},
   },
   date={1988},
   pages={35--68},
   review={\MR{1074589}},
}

\bib{dFH11}{article}{
   author={de Fernex, Tommaso},
   author={Hacon, Christopher D.},
   title={Deformations of canonical pairs and Fano varieties},
   journal={J. Reine Angew. Math.},
   volume={651},
   date={2011},
   pages={97--126},
   issn={0075-4102},
   review={\MR{2774312}},
   doi={10.1515/CRELLE.2011.010},
}


\bib{DS14}{article}{
   author={Donaldson, Simon},
   author={Sun, Song},
   title={Gromov-Hausdorff limits of K\"{a}hler manifolds and algebraic
   geometry},
   journal={Acta Math.},
   volume={213},
   date={2014},
   number={1},
   pages={63--106},
   issn={0001-5962},
   review={\MR{3261011}},
   doi={10.1007/s11511-014-0116-3},
}


\bib{DS16}{article}{
   author={Datar, Ved},
   author={Sz\'{e}kelyhidi, G\'{a}bor},
   title={K\"{a}hler-Einstein metrics along the smooth continuity method},
   journal={Geom. Funct. Anal.},
   volume={26},
   date={2016},
   number={4},
   pages={975--1010},
   issn={1016-443X},
   review={\MR{3558304}},
   doi={10.1007/s00039-016-0377-4},
}
	
%		
%\bib{dFEM10}{article}{
%   author={de Fernex, Tommaso},
%   author={Ein, Lawrence},
%   author={Musta\c{t}\u{a}, Mircea},
%   title={Shokurov's ACC conjecture for log canonical thresholds on smooth
%   varieties},
%   journal={Duke Math. J.},
%   volume={152},
%   date={2010},
%   number={1},
%   pages={93--114},
%   issn={0012-7094},
%   review={\MR{2643057}},
%   doi={10.1215/00127094-2010-008},
%}
%	
%
%\bib{dFKX17}{article}{
%   author={de Fernex, Tommaso},
%   author={Koll\'{a}r, J\'{a}nos},
%   author={Xu, Chenyang},
%   title={The dual complex of singularities},
%   conference={
%      title={Higher dimensional algebraic geometry---in honour of Professor
%      Yujiro Kawamata's sixtieth birthday},
%   },
%   book={
%      series={Adv. Stud. Pure Math.},
%      volume={74},
%      publisher={Math. Soc. Japan, Tokyo},
%   },
%   date={2017},
%   pages={103--129},
%   review={\MR{3791210}},
%}
%
%\bib{Fuj01}{article}{
%   author={Fujino, Osamu},
%   title={The indices of log canonical singularities},
%   journal={Amer. J. Math.},
%   volume={123},
%   date={2001},
%   number={2},
%   pages={229--253},
%   issn={0002-9327},
%   review={\MR{1828222}},
%}

\bib{Ful93}{book}{
   author={Fulton, William},
   title={Introduction to toric varieties},
   series={Annals of Mathematics Studies},
   volume={131},
   note={The William H. Roever Lectures in Geometry},
   publisher={Princeton University Press, Princeton, NJ},
   date={1993},
   pages={xii+157},
   isbn={0-691-00049-2},
   review={\MR{1234037}},
   doi={10.1515/9781400882526},
}

%
%\bib{FM18}{misc}{
%  author = {Filipazzi, Stefano},
%  author = {Moraga,Joaqu\'in}
%  title={Strong $(\delta,n)$-complements for semi-stable morphisms},
%  year = {2018},
%  note = {https://arxiv.org/abs/1810.01990},
%}
%
%\bib{Hac14}{article}{
%   author={Hacon, Christopher D.},
%   title={On the log canonical inversion of adjunction},
%   journal={Proc. Edinb. Math. Soc. (2)},
%   volume={57},
%   date={2014},
%   number={1},
%   pages={139--143},
%   issn={0013-0915},
%   review={\MR{3165017}},
%   doi={10.1017/S0013091513000837},
%}
%
%\bib{Hay99}{article}{
%   author={Hayakawa, Takayuki},
%   title={Blowing ups of $3$-dimensional terminal singularities},
%   journal={Publ. Res. Inst. Math. Sci.},
%   volume={35},
%   date={1999},
%   number={3},
%   pages={515--570},
%   issn={0034-5318},
%   review={\MR{1710753}},
%   doi={10.2977/prims/1195143612},
%}


\bib{GOST15}{article}{
   author={Gongyo, Yoshinori},
   author={Okawa, Shinnosuke},
   author={Sannai, Akiyoshi},
   author={Takagi, Shunsuke},
   title={Characterization of varieties of Fano type via singularities of
   Cox rings},
   journal={J. Algebraic Geom.},
   volume={24},
   date={2015},
   number={1},
   pages={159--182},
   issn={1056-3911},
   review={\MR{3275656}},
   doi={10.1090/S1056-3911-2014-00641-X},
}


\bib{Hay05a}{article}{
   author={Hayakawa, Takayuki},
   title={Divisorial contractions to 3-dimensional terminal singularities
   with discrepancy one},
   journal={J. Math. Soc. Japan},
   volume={57},
   date={2005},
   number={3},
   pages={651--668},
   issn={0025-5645},
   review={\MR{2139726}},
}
	
\bib{Hay05b}{article}{
   author={Hayakawa, Takayuki},
   title={Gorenstein resolutions of 3-dimensional terminal singularities},
   journal={Nagoya Math. J.},
   volume={178},
   date={2005},
   pages={63--115},
   issn={0027-7630},
   review={\MR{2145316}},
   doi={10.1017/S0027763000009120},
}
	
	
\bib{HK00}{article}{
   author={Hu, Yi},
   author={Keel, Sean},
   title={Mori dream spaces and GIT},
   note={Dedicated to William Fulton on the occasion of his 60th birthday},
   journal={Michigan Math. J.},
   volume={48},
   date={2000},
   pages={331--348},
   issn={0026-2285},
   review={\MR{1786494}},
   doi={10.1307/mmj/1030132722},
}	

\bib{HK10}{book}{
   author={Hacon, Christopher D.},
   author={Kov\'{a}cs, S\'{a}ndor J.},
   title={Classification of higher dimensional algebraic varieties},
   series={Oberwolfach Seminars},
   volume={41},
   publisher={Birkh\"{a}user Verlag, Basel},
   date={2010},
   pages={x+208},
   isbn={978-3-0346-0289-1},
   review={\MR{2675555}},
   doi={10.1007/978-3-0346-0290-7},
}
%
%\bib{HMX14}{article}{
%   author={Hacon, Christopher D.},
%   author={McKernan, James},
%   author={Xu, Chenyang},
%   title={ACC for log canonical thresholds},
%   journal={Ann. of Math. (2)},
%   volume={180},
%   date={2014},
%   number={2},
%   pages={523--571},
%   issn={0003-486X},
%   review={\MR{3224718}},
%   doi={10.4007/annals.2014.180.2.3},
%}
%
%
%\bib{HLS18}{misc}{
%  author = {Han, Jingjun},	
%  author = {Liu, Jihao},
%  author = {Shokurov, Vyacheslav V.},
%  title={Boundedness of Singularities admitting an $\epsilon$-PLT blow-up},
%  year = {2018},
%  note = {Work in progress},
%}


\bib{HX15}{article}{
   author={Hacon, Christopher D.},
   author={Xu, Chenyang},
   title={Boundedness of log Calabi-Yau pairs of Fano type},
   journal={Math. Res. Lett.},
   volume={22},
   date={2015},
   number={6},
   pages={1699--1716},
   issn={1073-2780},
   review={\MR{3507257}},
   doi={10.4310/MRL.2015.v22.n6.a8},
}

%
%\bib{Ish00}{article}{
%   author={Ishii, Shihoko},
%   title={The quotients of log-canonical singularities by finite groups},
%   conference={
%      title={Singularities---Sapporo 1998},
%   },
%   book={
%      series={Adv. Stud. Pure Math.},
%      volume={29},
%      publisher={Kinokuniya, Tokyo},
%   },
%   date={2000},
%   pages={135--161},
%   review={\MR{1819634}},
%}
%	
%\bib{IP01}{article}{
%   author={Ishii, Shihoko},
%   author={Prokhorov, Yuri},
%   title={Hypersurface exceptional singularities},
%   journal={Internat. J. Math.},
%   volume={12},
%   date={2001},
%   number={6},
%   pages={661--687},
%   issn={0129-167X},
%   review={\MR{1875648}},
%   doi={10.1142/S0129167X0100099X},
%}
%
	
\bib{Kol11}{misc}{
  author = {Koll\'ar, J\'anos},
  title={New examples of terminal and log canonical singularities},
  year = {2016},
  note = {https://arxiv.org/abs/1107.2864},
}

\bib{Kol13}{book}{
   author={Koll\'{a}r, J\'{a}nos},
   title={Singularities of the minimal model program},
   series={Cambridge Tracts in Mathematics},
   volume={200},
   note={With a collaboration of S\'{a}ndor Kov\'{a}cs},
   publisher={Cambridge University Press, Cambridge},
   date={2013},
   pages={x+370},
   isbn={978-1-107-03534-8},
   review={\MR{3057950}},
   doi={10.1017/CBO9781139547895},
}

%
%\bib{KK13}{article}{
%   author={Koll\'{a}r, J\'{a}nos},
%   author={Kov\'{a}cs, S\'{a}ndor J.},
%   title={Log canonical singularities are Du Bois},
%   journal={J. Amer. Math. Soc.},
%   volume={23},
%   date={2010},
%   number={3},
%   pages={791--813},
%   issn={0894-0347},
%   review={\MR{2629988}},
%   doi={10.1090/S0894-0347-10-00663-6},
%}

\bib{KM98}{book}{
   author={Koll\'{a}r, J\'{a}nos},
   author={Mori, Shigefumi},
   title={Birational geometry of algebraic varieties},
   series={Cambridge Tracts in Mathematics},
   volume={134},
   note={With the collaboration of C. H. Clemens and A. Corti;
   Translated from the 1998 Japanese original},
   publisher={Cambridge University Press, Cambridge},
   date={1998},
   pages={viii+254},
   isbn={0-521-63277-3},
   review={\MR{1658959}},
   doi={10.1017/CBO9780511662560},
}
%
%\bib{KMM87}{article}{
%   author={Kawamata, Yujiro},
%   author={Matsuda, Katsumi},
%   author={Matsuki, Kenji},
%   title={Introduction to the minimal model problem},
%   conference={
%      title={Algebraic geometry, Sendai, 1985},
%   },
%   book={
%      series={Adv. Stud. Pure Math.},
%      volume={10},
%      publisher={North-Holland, Amsterdam},
%   },
%   date={1987},
%   pages={283--360},
%   review={\MR{946243}},
%}
%		
%\bib{KX16}{article}{
%   author={Koll\'{a}r, J\'{a}nos},
%   author={Xu, Chenyang},
%   title={The dual complex of Calabi-Yau pairs},
%   journal={Invent. Math.},
%   volume={205},
%   date={2016},
%   number={3},
%   pages={527--557},
%   issn={0020-9910},
%   review={\MR{3539921}},
%   doi={10.1007/s00222-015-0640-6},
%}
%
%\bib{Li17}{article}{
%   author={Li, Chi},
%   title={K-semistability is equivariant volume minimization},
%   journal={Duke Math. J.},
%   volume={166},
%   date={2017},
%   number={16},
%   pages={3147--3218},
%   issn={0012-7094},
%   review={\MR{3715806}},
%   doi={10.1215/00127094-2017-0026},
%}

\bib{Li17}{article}{
   author={Li, Chi},
   title={K-semistability is equivariant volume minimization},
   journal={Duke Math. J.},
   volume={166},
   date={2017},
   number={16},
   pages={3147--3218},
   issn={0012-7094},
   review={\MR{3715806}},
   doi={10.1215/00127094-2017-0026},
}

\bib{LS13}{article}{
   author={Liendo, Alvaro},
   author={S\"{u}ss, Hendrik},
   title={Normal singularities with torus actions},
   journal={Tohoku Math. J. (2)},
   volume={65},
   date={2013},
   number={1},
   pages={105--130},
   issn={0040-8735},
   review={\MR{3049643}},
   doi={10.2748/tmj/1365452628},
}


\bib{LX16}{misc}{
  author ={Li, Chi}
  author = {Xu, Chenyang},
  title={Stability of Valuations and Kollár Components},
  year = {2016},
  note = {https://arxiv.org/abs/1604.05398},
}
		
\bib{LX17}{misc}{
  author ={Li, Chi}
  author = {Xu, Chenyang},
  title={Stability of Valuations: Higher Rational Rank},
  year = {2017},
  note = {https://arxiv.org/abs/1707.05561},
}


\bib{Mor18}{misc}{
  author = {Moraga, Joaqu\'in},
  title={On minimal log discrepancies and Koll\'ar components},
  year = {2018},
  note = {https://arxiv.org/abs/1810.10137},
}

%\bib{MP99}{article}{
%   author={Markushevich, D.},
%   author={Prokhorov, Yu. G.},
%   title={Exceptional quotient singularities},
%   journal={Amer. J. Math.},
%   volume={121},
%   date={1999},
%   number={6},
%   pages={1179--1189},
%   issn={0002-9327},
%   review={\MR{1719826}},
%}
%
%\bib{PS01}{article}{
%   author={Prokhorov, Yu. G.},
%   author={Shokurov, V. V.},
%   title={The first fundamental theorem on complements: from global to
%   local},
%   language={Russian, with Russian summary},
%   journal={Izv. Ross. Akad. Nauk Ser. Mat.},
%   volume={65},
%   date={2001},
%   number={6},
%   pages={99--128},
%   issn={1607-0046},
%   translation={
%      journal={Izv. Math.},
%      volume={65},
%      date={2001},
%      number={6},
%      pages={1169--1196},
%      issn={1064-5632},
%   },
%   review={\MR{1892905}},
%   doi={10.1070/IM2001v065n06ABEH000366},
%}
%
%\bib{PS09}{article}{
%   author={Prokhorov, Yu. G.},
%   author={Shokurov, V. V.},
%   title={Towards the second main theorem on complements},
%   journal={J. Algebraic Geom.},
%   volume={18},
%   date={2009},
%   number={1},
%   pages={151--199},
%   issn={1056-3911},
%   review={\MR{2448282}},
%   doi={10.1090/S1056-3911-08-00498-0},
%}
%

\bib{PS11}{article}{
   author={Petersen, Lars},
   author={S\"{u}ss, Hendrik},
   title={Torus invariant divisors},
   journal={Israel J. Math.},
   volume={182},
   date={2011},
   pages={481--504},
   issn={0021-2172},
   review={\MR{2783981}},
   doi={10.1007/s11856-011-0039-z},
}

%\bib{Sho92}{article}{
%   author={Shokurov, V. V.},
%   title={A supplement to: ``Three-dimensional log perestroikas'' [Izv.
%   Ross. Akad. Nauk Ser. Mat. {\bf 56} (1992), no. 1, 105--203; MR1162635
%   (93j:14012)]},
%   language={Russian},
%   journal={Izv. Ross. Akad. Nauk Ser. Mat.},
%   volume={57},
%   date={1993},
%   number={6},
%   pages={141--175},
%   issn={1607-0046},
%   translation={
%      journal={Russian Acad. Sci. Izv. Math.},
%      volume={43},
%      date={1994},
%      number={3},
%      pages={527--558},
%      issn={1064-5632},
%   },
%   review={\MR{1256571}},
%   doi={10.1070/IM1994v043n03ABEH001579},
%}
%
%\bib{Sho96}{article}{
%   author={Shokurov, V. V.},
%   title={$3$-fold log models},
%   note={Algebraic geometry, 4},
%   journal={J. Math. Sci.},
%   volume={81},
%   date={1996},
%   number={3},
%   pages={2667--2699},
%   issn={1072-3374},
%   review={\MR{1420223}},
%   doi={10.1007/BF02362335},
%}
%	
%
%\bib{Sho00}{article}{
%   author={Shokurov, V. V.},
%   title={Complements on surfaces},
%   note={Algebraic geometry, 10},
%   journal={J. Math. Sci. (New York)},
%   volume={102},
%   date={2000},
%   number={2},
%   pages={3876--3932},
%   issn={1072-3374},
%   review={\MR{1794169}},
%   doi={10.1007/BF02984106},
%}

\bib{Sho04}{article}{
   author={Shokurov, V. V.},
   title={Letters of a bi-rationalist. V. Minimal log discrepancies and
   termination of log flips},
   language={Russian, with Russian summary},
   journal={Tr. Mat. Inst. Steklova},
   volume={246},
   date={2004},
   number={Algebr. Geom. Metody, Svyazi i Prilozh.},
   pages={328--351},
   issn={0371-9685},
   translation={
      journal={Proc. Steklov Inst. Math.},
      date={2004},
      number={3(246)},
      pages={315--336},
      issn={0081-5438},
   },
   review={\MR{2101303}},
}


\bib{Tei99}{article}{
   author={Teissier, Bernard},
   title={Valuations, deformations, and toric geometry},
   conference={
      title={Valuation theory and its applications, Vol. II},
      address={Saskatoon, SK},
      date={1999},
   },
   book={
      series={Fields Inst. Commun.},
      volume={33},
      publisher={Amer. Math. Soc., Providence, RI},
   },
   date={2003},
   pages={361--459},
   review={\MR{2018565}},
}


\bib{Tzi05}{article}{
   author={Tziolas, Nikolaos},
   title={Three dimensional divisorial extremal neighborhoods},
   journal={Math. Ann.},
   volume={333},
   date={2005},
   number={2},
   pages={315--354},
   issn={0025-5831},
   review={\MR{2195118}},
   doi={10.1007/s00208-005-0676-9},
}

\bib{TX17}{article}{
   author={Tian, Zhiyu},
   author={Xu, Chenyang},
   title={Finiteness of fundamental groups},
   journal={Compos. Math.},
   volume={153},
   date={2017},
   number={2},
   pages={257--273},
   issn={0010-437X},
   review={\MR{3604863}},
   doi={10.1112/S0010437X16007867},
}

\bib{Wat81}{article}{
   author={Watanabe, Keiichi},
   title={Some remarks concerning Demazure's construction of normal graded
   rings},
   journal={Nagoya Math. J.},
   volume={83},
   date={1981},
   pages={203--211},
   issn={0027-7630},
   review={\MR{632654}},
}


\bib{Xu14}{article}{
   author={Xu, Chenyang},
   title={Finiteness of algebraic fundamental groups},
   journal={Compos. Math.},
   volume={150},
   date={2014},
   number={3},
   pages={409--414},
   issn={0010-437X},
   review={\MR{3187625}},
   doi={10.1112/S0010437X13007562},
}

\bib{Xu17}{misc}{
  author = {Xu, Chenyang},
  title={Interaction Between Singularity Theory and the Minimal Model Program},
  year = {2017},
  note = {https://arxiv.org/abs/1712.01041},
}

%\bib{Zar39}{article}{
%   author={Zariski, Oscar},
%   title={The reduction of the singularities of an algebraic surface},
%   journal={Ann. of Math. (2)},
%   volume={40},
%   date={1939},
%   pages={639--689},
%   issn={0003-486X},
%   review={\MR{0000159}},
%   doi={10.2307/1968949},
%}

\end{biblist}
\end{bibdiv}
\end{document}